\documentclass[10pt]{article}
\usepackage{graphicx} 
\usepackage{wrapfig}
\usepackage[margin=2.8cm]{geometry}
\usepackage{amsmath,amssymb,amsfonts,amsthm}
\usepackage{xcolor}
\usepackage{hyperref, cancel}
\usepackage[capitalize]{cleveref}
\usepackage{etoolbox}
\usepackage{mathtools}
\usepackage{bbm} 
\usepackage{xparse, expl3}
\usepackage{enumitem}
\usepackage{makecell}

\definecolor{inlink}{rgb}{0.1, 0, 0.8}
\definecolor{citelink}{rgb}{0, 0.4, 0.2}
\definecolor{outlink}{rgb}{0, 0.1, 0.7}
\hypersetup{
    colorlinks=true,
    linktoc=all,
    linkcolor=inlink, 
    citecolor=citelink, 
    urlcolor=outlink 
}

\apptocmd{\normalsize}{
  \setlength{\abovedisplayskip}{4pt plus 2pt minus 2pt}%
  \setlength{\belowdisplayskip}{4pt plus 2pt minus 2pt}%
  \setlength{\abovedisplayshortskip}{2pt plus 2pt minus 1pt}%
  \setlength{\belowdisplayshortskip}{2pt plus 2pt minus 1pt}%
}{}{}

\theoremstyle{plain}
\newtheorem{theorem}{Theorem}[section]
\newtheorem{proposition}[theorem]{Proposition}
\newtheorem{lemma}[theorem]{Lemma}

\newtheorem{corollary}[theorem]{Corollary}

\newtheorem{question}[theorem]{Question}

\theoremstyle{definition}
\newtheorem{definition}[theorem]{Definition}

\newtheorem{observation}[theorem]{Observation}

\newtheorem{example}[theorem]{Example}

\theoremstyle{remark}
\newtheorem{remark}[theorem]{Remark}

\DeclareMathOperator{\maxtrace}{maxtrace}
\DeclareMathOperator{\skel}{Skel}

\newcommand{\erdos}{Erd\H{o}s}
\newcommand{\real}{\mathbb{R}}
\newcommand{\Q}{\mathbb{Q}}
\newcommand{\smolu}{\mathbf{u}}
\newcommand{\smolv}{\mathbf{v}}
\newcommand{\allones}[2]{\mathbbm{1}_{#1\times #2}}

\newcommand*{\tr}[1][]{\operatorname{tr}_{#1}}
\newcommand*{\tp}[1]{#1^{\operatorname{T}}}
\newcommand*{\vect}[1]{\mathbf{#1}}
\DeclarePairedDelimiter{\Frobnorm}{\lVert}{\rVert_{\mathrm F}}
\newcommand*{\FrobIP}[2]{\left\langle#1, #2\right\rangle_{\mathrm F}}

\newcommand{\nn}{\nonumber}

\newcommand{\midmid}{\;\middle|\;}
\newcommand{\setbuilder}[3][]{\ifblank{#1}{\left\{}{#1\{}\,#2\ifblank{#3}{}{\ifblank{#1}{\midmid}{\;#1|\;}} #3\,\ifblank{#1}{\right\}}{#1\}}}
\newcommand{\setbuilderin}[2]{\{\,#1\ifblank{#2}{}{\,|\,} #2\,\}}

\newcommand{\keywords}[1]{%
  \begin{center}
    \textbf{Keywords:} #1
  \end{center}
}

\newcommand{\subjclass}[1]{%
  \begin{center}
    \textbf{2020 Mathematics Subject Classification:} #1
  \end{center}
}

\title{Characterization of Erd\H{o}s matrices by their zero entries}
\author{Priyanka Karmakar, Hariram Krishna, Souvik Pal and G. Krishna Teja}

\begin{document}

\maketitle

\begin{abstract}
    An Erd\H{o}s matrix $E$ is a bistochastic matrix whose sum of squares of entries (Frobenius norm squared) equals its maxtrace (maximum value of the trace of $\sigma E$, as $\sigma$ varies over permutation matrices).
    We characterize all Erd\H{o}s $E$ by the patterns of their zero entries; showing that each such skeleton has at most one $E$.
   We present an algorithm to find all $n\times n$ Erd\H{o}s matrices, which finds them up to $n\leqslant 5$ quickly and also size $n=6$. 
   We further show  some presently known RCDS matrices ($E$ in which the trace of $\sigma E$ remains constant across all the permutations that avoid every zero-entry position in $E$) to be Erd\H{o}s.
\end{abstract}


\section{Introduction}\label{intro}

Let $M_n(\mathbb{R})$ denote the set of all $n\times n$ matrices with their entries from real numbers $\mathbb{R}$.
A {\it bistochastic matrix} in $M_n(\mathbb{R})$ has all its entries in the interval $[0, 1]$, with each of its row sums and column sums equal to $1$.
These matrices form the {\it Birkhoff polytype} $\Omega_n$ -- also known as the {\it assignment polytope} -- and can be written as convex combinations of the set of {\it permutation matrices} $P_n$\cite{birkhoff1946tres}.

We recall the {\it Frobenius norm} $\Frobnorm{M}$ of a 
matrix $M$, which is the square root of the sum of squares of all its entries.
Observe $\Frobnorm{M}^2 = \sum_{i,j} M_{i, j}^2 = \tr(\tp{M}M)$; here $\tp{M}$ denotes the transpose of $M$.
The next notion is the trace of $M\in M_n(\mathbb{R})$ taken along a given permutation matrix $\sigma\in P_n$, defined as
\begin{equation}
    \tr[\sigma](M) := \tr(\tp{\sigma}M) = \tr(\tp{M}\sigma).
\end{equation}
This is also called the {\it $\sigma$-th diagonal sum} in $M$ in the literature.
As defined by Kushwaha and Tripathi \cite{kushwaha2025note} the {\it maxtrace} of $M$ is
\begin{equation}
    \maxtrace(M):= \max\limits_{\sigma\in P_n} \tr[\sigma](M).
\end{equation}

An observation of Marcus and Ree \cite{marcus1959diagonals} states that for any $M\in \Omega_n$, 
\[
\maxtrace(M) \geqslant \Frobnorm{M}^2
\]
\erdos{} asked for which bistochastic matrices $E\in \Omega_n$, the equality is attained;  which this paper is devoted to explore.
So, such $E$'s were called {\it \erdos{} matrices} by Tripathi \cite{tripathi2025erdos}; throughout, we reserve the symbol $E$ for them.

The easiest examples of \erdos{} matrices are all the permutation matrices.
One noteworthy property is that transposing $E$ or permuting its rows (or columns) preserves its maxtrace as well as its Frobenius norm and hence its \erdos-ness. 
In this paper, we are interested in the problem of finding all \erdos{} matrices $E$ up to the equivalences
$E\sim LER \sim L\tp{E}R$ for $L, R\in P_n$. 

The only such classes of $2\times 2$ \erdos{} matrices are the identity matrix $I_{2}$ and $J_2 = \frac{1}{2}\mathbbm{1}_{2\times 2}$ (where $\allones{l}{k}$ is an $l\times k$ matrix with all entries $1$).
Bouthat et al. \cite{bouthat2024question} initiated and studied $3\times 3$ \erdos{} matrices, found all of them and showed that they comprise of only 6 classes. 
Recently, Tripathi \cite[Theorems 1.3 \& 1.6]{tripathi2025erdos} established the following:
\begin{itemize}
\item[1)]
There are only finitely many \erdos{} matrices $E$ for each given dimension $n$.
\item[2)] All the entries of such $E$'s are rational. 
 This was proved by expressing the entries of such an $E$ as a solution to linear equations with integer coefficients; see Step 3 in Algorithm-1 in Section \ref{sec-algo}.
\end{itemize}
Then, \cite[Appendix A \& Example 4.3]{kushwaha2025note} characterized all $4\times4$ \erdos{} matrices, and also recorded an ``infinite \erdos{} array''.
One contribution of this paper is, finding all \erdos{} matrices of sizes $5\times 5$ and $6\times 6$ (see \Cref{Remark dim 5} and \Cref{Fig 6x6} in \Cref{Subsection applications}; and Python codes are available at Github repository \cite{github}).

Going back in timeline, the next important (but limited) series of examples are due to \cite[Corollary 2]{marcus1959diagonals}, $J_n = (1/n)\allones{n}{n}$ and notably, these are the only \erdos{} matrices with all entries positive. These also solve Van der Waerden’s permanental conjecture (i.e. they have the minimum permanent among all bistochastic matrices \cite{Erorychev}).
 Also, see \cite{Hwang, BrualdiMinper} for minimizers that solve the refined permanental conjecture within certain faces of $\Omega_n$
 Those minimizers in general are not \erdos{}.
\smallskip\\
We will set up some handy definitions to proceed further.
\begin{definition}\label{Definition Skeleton}
(a) {\bf Skeleton:}  For a matrix $M_{n\times m}$, its skeleton $\skel(M)$ is classically defined to be the following  binary $n\times m$-matrix
    \begin{equation}
        \skel(M)_{i, j} = \begin{cases}
            1 & \text{ if } M_{i,j}\neq 0,\\
            0 & \text{ if }M_{i,j}=0.
        \end{cases}
         \end{equation}
(b) {\bf Skeleton-poset:} We recall that the space of all $n\times m$ skeletons 
can be naturally endowed with a partial order. 
In that, two skeletons $S_1$ and $S_2$ are comparable as $S_1 \leqslant S_2$, iff $S_2 - S_1$ is nonnegative.\smallskip \\
(c)  {\bf Inner permutations:} 
We introduce the following notions.
For a given skeleton $S$, we define the set of inner permutations $P_n(S)$ to be all the permutation matrices whose 1's are also present in $S$:
\begin{equation}\label{Eqn inner permutations}
    P_n(S) := \setbuilder{\sigma \in P_n}{S \geqslant \sigma}.
\end{equation}
We shall call the traces along the inner permutations as {\it inner traces}. We shall call the complementary set of permutations as {\it outer permutations} (and correspondingly {\it outer traces}).
These notions can be extended to any matrix $M$ via $S=\skel(M)$ in the natural way. We shall use $P_n(M)=P_n(\skel(M))$.\end{definition}
\begin{example}\label{Example skeletons} For $M= \Big(\begin{smallmatrix}
    0.4 & \ 0.6 & 0\\
    0.6 & 0 & \ 0.4\\
    0 & \ 0.4 & \ 0.6
\end{smallmatrix}\Big)$, $\skel(M) = \Big(\begin{smallmatrix}
    1 & 1 & 0\\ 1 & 0 & 1\\ 0 & 1 & 1
\end{smallmatrix}\Big)$,
$P_3(M)= \left\{\Big(\begin{smallmatrix}
    0& 1 & 0\\
    1 & 0 & 0\\
    0 & 0& 1
\end{smallmatrix}\Big),\Big(\begin{smallmatrix}
    1& 0 & 0\\
    0 & 0 & 1\\
    0 & 1 & 0
\end{smallmatrix}\Big)\right\}$.
\end{example}
Two matrices have the same skeleton iff their zero positions are the same. 
We shall say that a skeleton has {\it total support} if there exists a bistochastic matrix with that skeleton \cite{Sinkhorn}.
Clearly, not every skeleton has total support; e.g. $(\begin{smallmatrix}
    1 & 1\\ 0 &1
\end{smallmatrix})$.
Recall that, faces of $\Omega_n$ are in bijection with skeletons with total support \cite{brualdi1977-1}. Our study of \erdos{} matrices takes us through such skeletons.
 
From a contemporary study by Brualdi and Dahl \cite{Brualdi2021} (and by our \cref{prop-maxtrace-everywhere}), we can see that the \erdos{} matrices are a special type of {\it RCDS matrices} recalled below.
\begin{definition}[{\cite{Brualdi2021}}]
A square matrix $M$ is called {\it Restricted Common Diagonal Sum} / RCDS matrix if (suggestive of the terminology): the diagonal sums $\tr[\sigma](M)$ remain constant across all $\sigma$ that avoid all 0's in $M$. 

In the terminology of this paper, all inner traces of an RCDS matrix are equal.
And this paper deals with RCDS matrices that are additionally bistochastic.
\end{definition}
As stated in \cite{Brualdi2021}, their inspirations to study RCDS bistochastic matrices are derived from  \cite{Achilles, Bala, Sinkhorndiagonals}.
Those three seminal papers characterized bistochastic matrices $M$'s with all traces (inner and some outer) other than some selected outer traces,  fixed constants.
Notably, Achilles \cite{Achilles} revealed poset structures on those $M$'s with respect to inclusions of outer permutations whose traces we do not fix.
Our main Theorem \ref{theorem-skeleton-uniqueness} completes this program, when those exempted traces in $M$'s are all the outer traces.
Classically, \cite{Achilles, Bala} focused on the properties of such RCDS $M$'s, including their stratification by zero-positions. 
However, the RCDS-ness property of \erdos{} matrices was not explored in the recent works \cite{bouthat2024question,tripathi2025erdos,kushwaha2025note} that recently initiated the study of \erdos{} matrices.

Our analysis in this paper begins by proving this RCDS-ness property of \erdos{} matrices in \cref{prop-maxtrace-everywhere}.
Conversely, any RCDS bistochastic matrix is \erdos{} if and only if the common inner traces are equal to its maxtrace (see \cref{cor-inner-trace-is-frob} and \cref{prop-maxtrace-everywhere}).
We note before passing-on that, not every totally supported skeleton is the skeleton of an RCDS matrix (in particular, of an \erdos{} matrix); see \cite[Example 1.5]{Brualdi2021} which was due to \cite{Achilles}.

Now we quote the following important result on the structure of RCDS bistochastic matrices.
\begin{theorem}[{\cite[Theorem 2.6]{Brualdi2021}}]\label{Theorem Brualdi-Dahl}
Given any RCDS bistochastic matrix $E$, we can express it as the following {\it Hadamard product} using its skeleton:
\begin{equation} \label{uv expression}
    E= \big[({\bf u}_i+{\bf v}_j)\times \skel(E)_{i,j}\big]_{1\leqslant i,j\leqslant n}
\end{equation}
for some real vectors  $\vect{u},\vect{v}$. 
These vectors $\vect{u} , \vect{v}$ are unique up to adding any fixed scalar $k_t$ to $I_t$-components in $\vect{u}$ and subtracting the same $k_t$ from $J_t$-components in $\vect{v}$, where $I_t$ and $J_t$ are set of indices such that $E_{I_t\times J_t}$ forms an indecomposable block of $E$, for each $t$.
\end{theorem}
We recall \cite{Brualdi2021} computed $\vect{u}, \vect{v}$ for each RCDS $E$ above, by solving 
\begin{equation}\label{Brualdi-Dahl system}
\begin{bmatrix}
  D_R  & \skel(E) \\
  \tp{\skel(E)} & D_C
\end{bmatrix} \left(\begin{matrix}
    \vect{u}\\ \vect{v}
\end{matrix}\right) = \allones{2n}{1},
\end{equation}
where $D_R$ is the diagonal matrix whose diagonal-entries are the row sums in $E$ and $D_C$ is the diagonal matrix of the column sums in~$E$.
This method indeed leads us to extend the rational-entries result on \erdos{} matrices by Tripathi: 
\begin{lemma}\label{Lemma rationality for RCDS}
    The entries of every RCDS bistochastic matrix are rational.
    In fact, there exists a choice of $\vect{u}$ and $\vect{v}$ which are in $\mathbb{Q}^n$ when expressed as in \eqref{uv expression}.
\end{lemma}
This lemma is proved in Subsection~\ref{mainproof}. 
Now we turn our focus onto \erdos{}-ness property, beginning with a criterion for determining it:
\begin{lemma} \label{rcds to erdos}
    Let $E= \big[(\smolu_i+\smolv_j)\times \skel(E)_{i,j}\big]_{1\leqslant i,j\leqslant n}$ be an RCDS bistochastic matrix. Then $E$ is \erdos{} if and only if for all outer permutations $\sigma\in P_n\setminus P_n(E)$, the (outer) sums
    \begin{equation}
        \sum\limits_{\substack{1\leqslant i, j \leqslant n\\ \sigma_{i, j}=1\  \&\  E_{i, j}=0}}(\smolu_i + \smolv_{j}) \geqslant 0.
    \end{equation}
\end{lemma}
Hence, $\min(\vect{u}) + \min(\vect{v}) \geqslant 0$ is a sufficient condition for $E$ to be \erdos{}.
This lemma is proved in Subsection~\ref{mainproof}.
Our first result generates more families of \erdos{} matrices following \cite{Brualdi2021}.
 Brualdi and Dahl show RCDS-property of the following three sets of matrices, notably in arbitrary dimensions $n$.
\begin{enumerate}
\item \hypertarget{zigzag}{} For every positive integer triad $0<s<r<n$,
\[
X^{(r,s,n)}:= \left(\begin{array}{@{}c|c@{}}
    \frac{1}{r} \allones{r}{s} & \frac{r-s}{r(n-s)}\allones{r}{(n-s)}  \\ \hline
  {\bf 0}_{(n-r)\times s}   & \frac{1}{n-s}\allones{(n-r)}{(n-s)}
\end{array}\right)_{n\times n}.
\]
\item {\it Zig-zag} RCDS patterns: Fix two positive integer tuples $\vect{r}=(r_1,\ldots, r_k)$ and $\vect{s}=(s_1,\ldots, s_{k+1})$,
satisfying
the ``dominance ordering'' and interlacing conditions: 
\begin{eqnarray*}
    \sum\limits_{i=1}^t s_i < \sum\limits_{i=1}^t r_i < \sum\limits_{i=1}^{t+1}s_i \ \forall \ 1\leqslant t < k &\quad \text{ and} &\quad\sum_{i=1}^{k}s_i\leqslant\sum_{i=1}^{k}r_i = \sum_{i=1}^{k+1}s_i = n. 
    \end{eqnarray*}
  
\begin{equation} \label{brualdi family 3}
  \text{Consider} \ X^{(\vect{r}, \vect{s})}:=\left(\begin{array}{@{}c|c@{}|c@{}|c@{}|@{}c}
    \alpha_{1, 1} \allones{r_1}{s_1} & \alpha_{1, 2} \allones{r_1}{s_2} & 0 &  & 0 \\ \hline
    0 & \alpha_{2, 2} \allones{r_2}{s_2} & \alpha_{2, 3} \allones{r_2}{s_3} &  & 0\\ \hline
    & & \ddots &  \ddots &   \\ \hline
    0 & 0 & 0 & \alpha_{k, k} \allones{r_k}{s_k} & \  \alpha_{k, k+1} \allones{r_{k}}{s_{k+1}} 
\end{array}\right)_{n\times n}.
\end{equation}
Here $\alpha_{i, j}>0$ (see \cref{Remark Brualdi's row size mismatch})
We shall allow $s_{k+1}=0$, i.e. $k+1$-th column block to be void. This family includes the family in point 1.

\item Fix $p\in \mathbb{N}$, $n=2p$, and any $4$-tuple $\vect{\alpha}=(\alpha_1,\ldots, \alpha_4)\in \{1,\ldots, p\}^4$ satisfying $\alpha_1+\alpha_4=\alpha_2+\alpha_3$.
Next, for each $1\leqslant i\leqslant 4$, we consider a $p\times p$-skeleton $A_i$ with each row and column of $A_i$ having exactly $\alpha_i$-many 1's.
We define $X^{\vect{\alpha}}:= \frac{1}{\alpha_1\alpha_4 + \alpha_2\alpha_3}\left(\begin{array}{@{}c|c@{}}
    \alpha_4A_1 & \alpha_3A_2\\ \hline
  \alpha_2A_3& \alpha_1A_4 
\end{array}\right)_{n\times n}$.

\end{enumerate}
\begin{theorem}\label{Theorem Erdos-ness of 3 families of RCDS}
    All RCDS matrices in families 1 ($X^{(r,s,n)}$) and 3 ($X^{\vect{\alpha}}$) above are \erdos{}. Additionally, all the matrices in family 2 ($X^{(\vect{r}, \vect{s})}$) with $k\leqslant 2$ are also \erdos{}. There exist RCDS matrices $X^{(\vect{r}, \vect{s})}$ that are not \erdos{} when $k\geqslant 3$.
\end{theorem}
This theorem is proved in \Cref{Section proof for RCDS patterns}.
See \cref{Remark Brualdi's row size mismatch} on the construction of $X^{(\vect{r}, \vect{s})}$'s from \cite{Brualdi2021}. 
See \Cref{type 3 counter} for a counterexample of a non-\erdos{} matrix of the type $X^{(\vect{r}, \vect{s})}$ for $k = 3$.

To head to our second result, we resume our discussion on \cite[Algorithm 1]{tripathi2025erdos} which produces a possibly \erdos{} matrix $E$ for each linearly independent subset of permutations in $P_n$. Indeed, the output matrices of that algorithm also include all RCDS matrices.
Our paper improves upon this by mapping the classically well-studied binary matrices with total support onto a set which includes all \erdos{} matrices. 
Indeed, the count of totally supported skeletons is smaller than that of linearly independent subsets (on which Tripathi's study is based) of $P_n$, see \cref{Remark comparing Algorithms}. 
In our next result \cref{theorem-skeleton-uniqueness}b, we show that there is at most one \erdos{} matrix with a given skeleton.
Equivalently, there is at most one \erdos{} matrix in each {\it proper face} of the polytope $\Omega_n$.
Thus, one might be able to study and characterize \erdos{} matrices more closely with the abundance of literature (\cite{Sinkhorn} and \cite{brualdi1977-1, brualdi1977-2, brualdi1977-3, brualdi1976-4} among several others) available on skeletons/binary/$\{0,1\}$-matrices and faces of $\Omega_n$.

Now we are ready to state our main result which will be proved in \cref{mainproof}.
Notably, this strengthens \cite[Theorem 2]{Achilles} by showing the strict inequality in statement a).
\begin{theorem} \label{theorem-skeleton-uniqueness}
  Fix two \erdos{} matrices $E_1, E_2$.\begin{enumerate} [label=\alph*)]
        \item If $\skel(E_1) < \skel(E_2)$, then $\maxtrace(E_1)>\maxtrace(E_2)$.
        \item If $\skel(E_1) = \skel(E_2)$, then $E_1=E_2$.
    \end{enumerate}
\end{theorem}
\begin{remark} \label{remark-count}
    The number of \erdos{} matrices of size $n\times n$ is thus bounded by $2^{n^2}$.
\end{remark}
We prescribe an algorithm to find all such matrices in \cref{sec-algo} utilizing these insights.
We found all $6\times 6$ \erdos{} matrices using this algorithm. Given that there can be up to $2(n!)^2$ elements in an equivalence class and the upper bound from \cref{remark-count}, a reasonable estimate for the count of equivalence classes of \erdos{} matrices is ${2^{n^2}} / {2(n!)^2}$. This visibly follows from \cref{tab-count} given below.
\begin{table}[h]
    \centering
    \begin{tabular}{|c|r|r|r|r|}
        \hline
        $n$ & \makecell{A: Count of $n\times n$\\ \erdos{} classes}  & \makecell{B: Binary matrices\\ with total support} & A/B
        & A/$\frac{2^{n^2}}{2(n!)^2}$
        \\ \hline
        1 & 1 & 1 & 1 & 1\\
        2 & 2 & 2 & 1 & 1\\
        3 & 6 & 6 & 1 & 0.84\\
        4 & 32 & 33 & 0.97 & 0.56\\
        5 & 469 & 534 & 0.89 & 0.40\\
        6 & 23851 & 32174 & 0.74 & 0.36\\ \hline
    \end{tabular}
    \caption{Count of $n\times n$ \erdos{} matrices up to transposition and permutations, compared to total number of binary matrices with total support up to the same equivalences.}
    \label{tab-count}
\end{table}

The sequence of the number of $n\times n$ binary matrices with total support is available in \cite{oeisA326342} .
More relevant here is \cite{oeisA326343} which lists the number of equivalence classes of such matrices up to the permutation of rows and columns. 
This refinement gives us an upper bound on the number of equivalence classes of \erdos{} matrices.
Our work can currently extend those integer sequences, as well as the \erdos{} matrix count sequence \cite{oeisA381896} by the terms 791 and 46185 for $n=5, 6$.

In Subsection \ref{Subsection erdos denominators}, we discuss the matrices with most number of unique nonzero entries and those with the largest denominator when expressed in their simplest terms.

\section{Uniqueness of \erdos{} matrices by their zeros}\label{Section 2}
We recall the Birkhoff-von Neuman Theorem here-
\begin{theorem}[\cite{birkhoff1946tres}]\label{Theorem by Birkoff}
Given any $M\in \Omega_n$ the space of $n\times n$ bistochastic matrices, we can express it as a convex sum of the permutation matrices.
\begin{equation} \label{birkhoff--convex-sum}
    M = \sum_{\sigma_i\in P_n}{a_i\sigma_i};\quad \quad
    \sum a_i=1, \quad \quad
    0\leqslant a_i \leqslant 1.
\end{equation}
\end{theorem}
 We wish to recast this result in the framework of skeletons. This will be more useful for the methods of this paper.
Let us denote the vector subspace in $M_n(\real)$ spanned by all the matrices of $P_n(M)$ by $\real P_n(M)$; it might suffice to work just over the field of rationals $\Q$.
\begin{observation}
 $M\in \real P_n(M)$ by \cref{Theorem by Birkoff}.
\end{observation}
\begin{corollary} \label{cor-birkoff}
    Given $M\in \Omega_n$ and $\skel(M) = S$, there exist $a_i$ and $\sigma_i$ such that
    \begin{equation} \label{birkhoff--skeleton-convex-sum}
    M = \sum_{\sigma_i\in P_n(M)}{a_i\sigma_i};\quad \quad
    \sum a_i=1, \quad \quad
    0\leqslant a_i \leqslant 1.
\end{equation}
\end{corollary}
The {\it Birkoff's algorithm} expresses $M\in \Omega_n$ as a convex sum of linearly independent permutation matrices in $P_n(M)$;
so $\sigma_i$ in \eqref{birkhoff--convex-sum} or \eqref{birkhoff--skeleton-convex-sum} with nonzero coefficients $a_i$ can be assumed to be linearly independent.
\begin{lemma} \label{lemma-inner-always}
    Given any $M\in\Omega_n$, and any permutation $\sigma_0\in P_n(M)$, there exists a collection $\{\sigma_1, \ldots \sigma_l\}\subseteq P_n(M)\setminus\{\sigma_0\}$ and $a_0, \ldots, a_{l}\in [0,1]$, such that $a_0 > 0$ and $M = \sum_{0\leqslant i\leqslant l}a_i\sigma_i$.
\end{lemma}
In other words, a bistochastic matrix always has a convex-sum expansion that involves any permutation of our choice contained within its skeleton.
\begin{proof}
    We will first create a matrix $G:= (1+\epsilon)M - \epsilon\sigma_0$ with a small value of $\epsilon>0$. Since $M$ is nonzero at every position where $\sigma_0$ is nonzero we can choose an $\epsilon$ such that $G$ is a bistochastic matrix.
    In particular, any value $0 < \epsilon \leqslant \min_{j}{M_{j,\sigma_0(j)}}$ will do.
    Now by \cref{cor-birkoff}, $G$ can be expressed as\medskip\\
    \hspace*{3cm}    $G = \sum\limits_{\sigma_i\in P_n(M)}{a_i\sigma_i}; \quad\quad
        M = \frac{1}{1+\epsilon}
            \Big(
                \epsilon \sigma_0 + \sum\limits_{\sigma_i\in P_n(M)}{a_i\sigma_i}
            \Big)$.
\end{proof}
\begin{corollary}
More generally, every $M\in \Omega_n$ admits a convex-sum expression over all the permutations in $P_n(M)$ appearing simultaneously.
    To see this, iteratively apply \cref{lemma-inner-always} to every $\sigma\in P_n(E)$ and take the average over all the expressions.
\end{corollary}
Next, we would like to know along which permutations, the maxtrace for an \erdos{} matrix is attained.
For this, we will recall the following lemma from \cite[Equation (2.1)]{marcus1959diagonals} but specialized for \erdos{} matrices.
We reprove it here as our proof for \cref{theorem-skeleton-uniqueness} has a similar construction.
\begin{lemma} \label{lemma-inner-maxtrace}
    Let $E$ be an \erdos{} matrix. Suppose $\setbuilderin{\sigma_1, \ldots, \sigma_l}{}\subseteq P_n(E)$ be such that we can express
    $E = \sum_{1\leqslant i\leqslant l}{a_i\sigma_i};\quad
    \sum a_i=1, \quad
    0 < a_i \leqslant 1$, then
    \begin{equation}
        \maxtrace(E) = \tr[\sigma_i](E) \quad \forall\ i.
    \end{equation}
\end{lemma}
\begin{proof}
    We can express
    \begin{align}
        \maxtrace(E) = \Frobnorm{E}^2 = \tr(\tp{E}E) = \sum_{1\leqslant i\leqslant l}\tr{(a_i\tp{\sigma_i}E)} = \sum_{1\leqslant i\leqslant l}{a_i \tr[\sigma_i](E)}. \label{maxtrace-where}
    \end{align}
    Here, each $\tr[\sigma_i](E) \leqslant \maxtrace(E)$ and $\sum{a_i}=1$.
    So, the final expression in \eqref{maxtrace-where} is at most $\maxtrace(E)$. This maximum is attained if only if $\tr[\sigma_i](E) = \maxtrace(E)$ for all $i$.
\end{proof}

\begin{proposition} \label{prop-maxtrace-everywhere}
Suppose $E$ is \erdos{}.
$\maxtrace(E) = \tr[\sigma](E)$ $\forall$ $\sigma\in P_n(E)$.
So $E$ is an RCDS matrix.
\end{proposition}
This is easily seen by putting together \cref{lemma-inner-maxtrace,lemma-inner-always}.
\begin{corollary} \label{cor-inner-trace-is-frob}   
All the inner traces in an \erdos{} $E$ are the same and equal to $\Frobnorm{E}^2$. 
This result extends to RCDS bistochastic matrices as well, by the equality in \eqref{maxtrace-where} between $\|E\|_{\mathrm{F}}^2$ and the last sum $\sum\limits_{1\leqslant t\leqslant l}a_i\tr[\sigma_i](E)= \Big(\sum\limits_{1\leqslant i\leqslant l}a_i\Big)\tr[\sigma_1](E)= \tr[\sigma_i](E)$ for all $1\leqslant i\leqslant l$.    
\end{corollary}
In an \erdos{} matrix, note that an outer permutation may or may not yield the maxtrace.


\subsection{Proofs of \texorpdfstring{\cref{theorem-skeleton-uniqueness}; Lemmas \ref{Lemma rationality for RCDS}, \ref{rcds to erdos}}{main Theorem and lemma}} \label{mainproof}
We shall use $\FrobIP{.}{.}$ for the Frobenius inner product.
\begin{proof}[Proof of \cref{theorem-skeleton-uniqueness}]
 Part a) : Fix two \erdos{} matrices $E_1<E_2$. Let $\{\sigma_1,\ldots, \sigma_k\}\subseteq P_n(E_1)$ be a basis for $\real P_n(E_1)$ and by basis extension theorem, let $\{\sigma_{k+1}, \ldots, \sigma_m\}\subseteq  P_n(E_2)\setminus P_n(E_1)$ be a set of permutations such that $\{\sigma_1,\ldots, \sigma_m\}$ is a basis for $\real P_n(E_2)$.
  By \cref{cor-birkoff}, let $E_1= \sum_{i=1}^{k}c_i\sigma_i$ and $E_2=\sum_{j=1}^{m}d_j\sigma_j$ for some $c_i$ and $d_j\in \Q$. Next recall that $\tr[\sigma_j](E_1)\leqslant\maxtrace(E_1)=\Frobnorm{E_1}^2$ for $k<j\leqslant m$. Now using this and \cref{prop-maxtrace-everywhere}
  \begin{align*}
    \Frobnorm{E_1}^2&= \sum_{i=1}^k\FrobIP{c_i\sigma_i}{E_1}
    =\sum_{i=1}^{k}c_i\tr[\sigma_i](E_1) \\
    &\geqslant \sum_{j=1}^{m}d_j\tr[\sigma_j](E_1) \quad\quad \big(\text{as  $\tr[\sigma_1](E_1)= \ldots = \tr[\sigma_k](E_1)=\maxtrace(E_1)$ $\geqslant \tr[\sigma_{i}](E_1)$ for $k < i \leqslant m$}\big)\\
    &= \sum_{i=1}^{k}\sum_{j=1}^{m}c_id_j\FrobIP{\sigma_i}{\sigma_j}
    = \sum_{i=1}^{k} c_i\FrobIP{\sigma_i}{E_2}
    = \Frobnorm{E_2}^2\qquad \big(\text{as $E_2$ is \erdos{}}\big).
  \end{align*}
  To show the strict inequality, suppose on the contrary that $\maxtrace(E_1)=\maxtrace(E_2)$. 
  Then $\tr[\sigma](E_2)= \maxtrace (E_2)=\maxtrace(E_1)$ $\forall$ $\sigma\in P_n(E_2)$ by Proposition \ref{prop-maxtrace-everywhere}.
  We turn to computations using the Gram matrix $G = G(\sigma_1,\ldots, \sigma_m)= \big[\FrobIP{\sigma_i}{\sigma_j}\big]_{1\leqslant i,j\leqslant m}$, which is non-singular.
  Now solving for the $m\times 1$ vector $\vect{x}$ in the equation $G \vect{x} = \maxtrace(E)\allones{m}{1}$ leads to the two solutions $\vect{x}= \tp{(d_1,\ldots, d_m)}$ and $\vect{x} = \tp{(c_1,\ldots, c_k, 0,\ldots, 0)}$ (see \cite[Algorithm 1]{tripathi2025erdos}).
  The non-equality of these two solutions follows from the non-equality of the skeletons of $E_1, E_2$, which contradicts the non-singularity of $G(\sigma_1,\ldots, \sigma_m)$.
\smallskip\\
 Part b) : Let $\{\sigma_1,\ldots, \sigma_m\}\subseteq P_n(E_1)$ be a basis for $\real P_n(E_1)=\real P_n(E_2)$, $E_1=\sum c_i\sigma_i$, $E_2=\sum d_i\sigma_i$ and $G=\big[\FrobIP{\sigma_i}{\sigma_j}\big]_{1\leqslant i,j\leqslant m}$. By repeated use of \cref{prop-maxtrace-everywhere} and \Cref{cor-inner-trace-is-frob}, we see that
  \begin{equation}
    \begin{aligned}\nn
  &      \Frobnorm{E_1}^2 = \tr(\tp{E_1}E_1) = \sum c_i \langle \sigma_i, E_1 \rangle_{\mathrm{F}} =  \sum c_i\tr[\sigma_i](E_1)=\sum c_i \|E_1\|_{\mathrm{F}}^2 = \sum d_i \|E_1\|_{\mathrm{F}}^2 = \sum d_i\tr[\sigma_i](E_1) \hspace*{0.5cm} \\ & \hspace*{0.5cm} 
        = \sum d_i \langle  \sigma_i, E_1\rangle_{\mathrm{F}} =  \tr(\tp{E_2}E_1)  = \tr(\tp{E_1}E_2) = \sum c_i\langle \sigma_i, E_2  \rangle_{\mathrm{F}}= \sum c_i \mathrm{tr}_{\sigma_i}(E_2)= \sum c_i \|E_2\|_{\mathrm{F}}^2= \|E_2\|_{\mathrm{F}}^2.
    \end{aligned}
  \end{equation}
  We observe, both $\tp{(d_1,\ldots, d_m)}$ and $\tp{(c_1,\ldots, c_m)}$ to satisfy $G\vect{x}=\Frobnorm{E_1}^2\allones{m}{1}$. 
  The non-singularity of a Gram matrix for linearly independent $\{\sigma_1,\ldots, \sigma_m\}$ forces $(c_1,\ldots, c_m)=(d_1,\ldots, d_m)$ and so $E_1=E_2$.
\end{proof}
\noindent{The above proof also works for RCDS matrices and can be an alternative proof for \cite[Corollary 1.3]{Brualdi2021}.}
\begin{remark}
    The maxtrace (also squared Frobenius norm) map on the subset of \erdos{} matrices in $\Omega_n$ attains its minimum of $1$ only at $J_n$, and similarly, its maximum of $n$ only at every $\sigma\in P_n$. Notably, if \erdos{} $E$ has a zero entry, i.e. $\skel(E) < \allones{n}{n}$,  necessarily $\tr[\sigma](E)< \maxtrace(E)$ for some outer  $\sigma\in P_n\setminus P_n(E)$. 
\end{remark}
\begin{proof}[Proof of Lemma \ref{Lemma rationality for RCDS}]
Let $H$ be the matrix of coefficients in the system of linear equations in \eqref{Brualdi-Dahl system} for computing $\vect{u}, \vect{v}$.
Note that $H$ has integer entries.
 It is stated in \cite[Proof of Theorem 2.6]{Brualdi2021}, that $H$ is a sign-less Laplacian matrix, with 1-dimensional null-space spanned by $(\allones{n}{1}, \ - \allones{n}{1})$.
 In particular as $H$ is Hermitian, the image subspace $H\mathbb{R}^{2n}$ is the orthogonal compliment of $\mathbb{R}(\allones{n}{1}, \ - \allones{n}{1})$, which is $2n-1$ dimensional.
 Indeed (as rank of $H$ is fixed across working over $\mathbb{Q}\subset \mathbb{R}$), the same assertion is true when we work over rationals $\mathbb{Q}$, i.e. $H\mathbb{Q}^{2n}$ is the orthogonal compliment of $\mathbb{Q}(\allones{n}{1}, \ - \allones{n}{1})$ inside $\mathbb{Q}^{2n}$; with respect to the restriction of the usual Euclidean inner product.
 Finally as $\allones{2n}{1}$ lies in that complementary subspace, we can have a rational solution $({\bf u}, {\bf v})$ to the desired system.
\end{proof}

\begin{proof}[Proof of \cref{rcds to erdos}]
Fix an RCDS bistochastic $E= \big[({\bf u}_i+{\bf v}_j)\times \skel(E)_{i,j}\big]_{1\leqslant i,j\leqslant n}$ for some real vectors ${\bf u}, {\bf v}$.
For any inner $\sigma\in P_n(E)$, observe that $\tr[\sigma](E)=\|E\|_{\mathrm{F}}^2=\sum\limits_{i=1}^n ({\bf u}_i+ {\bf v}_{\sigma(i)})= \sum\limits_{i=1}^n {\bf u}_i + \sum\limits_{i=1}^n {\bf v}_{\sigma(i)}= \sum\limits_{i=1}^{n}({\bf u}_i+{\bf v}_i)$.
For an outer $\tau\in P_n\setminus P_n(E)$, observe that $\tr[\tau](E) = \sum\limits_{1\leqslant i\leqslant n:\ E_{i, \tau(i)}>0} {\bf u}_i+{\bf v}_{\tau(i)}
= \|E\|_{\mathrm{F}}^2 - \sum\limits_{1\leqslant i\leqslant n:\ E_{i, \tau(i)}=0} {\bf u}_i+{\bf v}_{\tau(i)} 
$.
If that second sum is  non negative for every outer $\tau$, then $\tr[\tau](E)\leqslant\|E\|_{\mathrm{F}}^2=\maxtrace(E)$ as required.
\end{proof}

\section{An algorithm to find all \erdos{} matrices} \label{sec-algo}
We develop skeleton-based Algorithm-1 to find all $n\times n$ \erdos{} matrices, generalizing \cite[Algorithm 1]{tripathi2025erdos}.
An implementation of it is available at \cite{github}.
By \cref{theorem-skeleton-uniqueness}b, there can be at most one \erdos{}
\noindent matrix  with a given skeleton. 
As we are only interested in their equivalence classes up to transpositions and left/right multiplication by permutations, we start by finding the representative skeletons as follows:

\subsection{Algorithm-1}\label{Subsection Algorithm}

{\bf Step 1:} Prepare a list of all $n\times n$ skeletons called $B_n$. There are $2^{n^2}$ of them. We will sieve through it and prepare a shorter list $S_{repr}$. For each matrix $B$ from the list $B_n$, we first append it to $S_{repr}$. Then compute all the matrices $LBR$ and $L\tp{B}R$ for $L, R\in P_n$ and remove them from $B_n$. Repeat this process until we have exhausted all the entries in $B_n$ and completely populated $S_{repr}$ with equivalence classes.

\noindent{\bf Step 2:} Next, we shall prune the list $S_{repr}$ for skeletons with total support(defined below Example \ref{Example skeletons}). 
\begin{equation}
   S_{ts} := \setbuilder[\Big]{S\in S_{repr}\setminus(0)_{n\times n}}
    {\skel\Big(\sum\nolimits_{\sigma\in P_n(S)} \sigma\Big) = S}.
\end{equation}

\noindent {\bf Step 3:} Now for a given skeleton $S\in S_{ts}$, within the set of permutation matrices $P_n(S)$,
we find a maximal linearly independent subset of permutations $B_S$ for $S$. This basis spans the linear space $\mathbb{R}P_n(S)$, but we focus only on sums $\sum_{\sigma_i\in B_S} c_i \sigma_i$ for $\sum c_i = 1$. Note that we shall allow individual $c_i < 0$.

The next step will be akin to Algorithm 1 in \cite{kushwaha2025note}

\noindent {\bf Step 4:}
For our basis set $B_S = \{\sigma_1, \ldots, \sigma_l\}$, we compute the gram matrix $M$ given by $M_{ij} = \FrobIP{\sigma_i}{\sigma_j}$. We then solve for the $l\times1$ column vector $\mathbf{y}$ satisfying $M\mathbf{y}=\allones{l}{1}$. We normalize this to
$\mathbf{x} = \mathbf{y}/\sum{y_i}$. Now we compute $E_S = \sum{x_i\sigma_i}$.

\noindent {\bf Step 5:} We check for the \erdos-ness of $E_S$ obtained in the previous step. Namely, whether:

\begin{itemize}
    \item $E_S$ has no negative entries.
    \item $\maxtrace(E_S) = \Frobnorm{E_S}^2$.
\end{itemize} 
The simplest way to find the maxtrace is to compute all the $n!$ traces.
On the other hand, one could use Hungarian algorithm which efficiently solves the {\it assignment problem} in $O(n^3)$ steps/time.
Notably, Step 4 already ensures that all the inner traces of $E_S$ are the same and equal to the squared Frobenius norm $\Frobnorm{E_s}^2$.
Finally, we gather all those $E_S$ that have the above properties, while also discarding the cases of $\skel(E_S) \neq  S$ (which happens for \cite[Example 1.5]{Brualdi2021}) to avoid duplication.

\subsection{Features and applications of Algorithm 1} \label{Subsection applications}
\begin{remark}[Better bounds]\label{Remark comparing Algorithms}
  Theoretically, in each dimension $n$,   \cite[Algorithm 1]{tripathi2025erdos} iterates through all the subsets of $P_n$ of size $ \leqslant\dim\Omega_n+1$ ($=(n-1)^2+1$). Their count is  $\sum\nolimits_{2\leqslant j\leqslant (n-1)^2+1 }\binom{n!}{j}$ which is essentially more than the order of $(n!)$.
  We improve on this as the count of $n\times n$ skeletons is $2^{n^2}$. 
\end{remark}
\begin{remark}[Fixing an initial permutation] \label{favourite-id}
We can fix a chosen permutation matrix $\sigma_0$ and we know that for any \erdos{} matrix $E$, at least one of the equivalent representatives of $E$ contains $\sigma_0$ in its skeleton, i.e., there exists $L, R\in P_n$ such that $\sigma_0 \in P_n(LER)\cup P_n(L\tp{E}R)$.
We can use this fact to speed up Step 1 in the above algorithm by starting with the list of binary matrices that contain $\sigma_0$. As we are fixing $n$ positions to be 1, the list now has its size $=2^{n^2-n}$ instead.
\end{remark}

\begin{remark}[Non-positive solutions $\bf{y}$] In contrast to \cite{kushwaha2025note}, in Step 4 we do not discard the cases where $\mathbf{y}$ has negative entries. This is because the even if $\mathbf{y}\leqslant 0$, the resultant matrix $E_S$ can still be bistochastic.
We only discard the cases when some entry of $E_S$ becomes negative in the subsequent step.
\end{remark}

\begin{remark}[Basis choice]
It suffices to choose any fixed basis in Step 3. 
If there existed an \erdos{} matrix $E_S$ for a given skeleton $S$,  then it is necessarily a sum $E_S = \sum c_i \sigma_i$ over the basis $\{\sigma_1, \ldots, \sigma_l\} = B_S$  fixed in Step 3.
The vector $\mathbf{x} = (c_i)_{l\times 1}$ is a solution for $M\mathbf{x} = m\allones{l}{1}$ in Step 4 for some quantity $m>0$.
\end{remark}

\begin{remark}[Algorithm selection] The methods of \cite{Brualdi2021} might be more efficient than the one in \cite{tripathi2025erdos} our Step 4 used. 
But, for its proper implementation, we need to break down decomposable matrices into indecomposable blocks and combine them again later. 
So, we followed \cite{tripathi2025erdos} for the sake of simplicity.
\end{remark}

\begin{remark}[Some statistics about non-examples in dimension $n=5$]\label{Remark dim 5}
Step 1 evaluates 3014 equivalence classes of binary matrices (or 1764 in view of  \cref{favourite-id}). This gets pruned down to 534 skeletons with total support after Step 2. From these only 469 distinct \erdos{} matrices are obtained.

In 9 classes, the matrix $E_S$ obtained in Step 4 has a strictly smaller skeleton than $S$. This happens whenever an outer trace is equal to the inner traces and square of the Frobenius norm. Then the skeleton created by the inclusion of those outer permutations would yield the same result as the smaller skeleton.
\end{remark}

\begin{wrapfigure}{r}{0.4\textwidth}
    \vspace{-20pt}
    \includegraphics[scale=0.58]{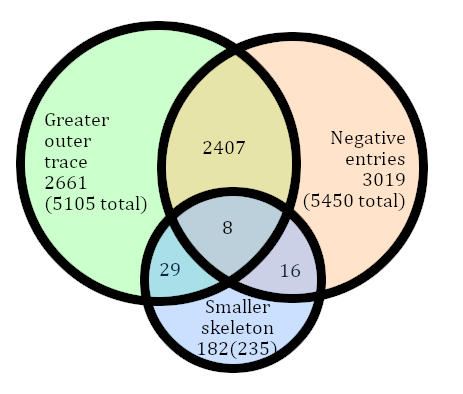}
    \vspace*{-20pt}
      \caption{Counts for $6\times 6$ skeletons each of which lack an \erdos{} matrix by reason.}\label{Fig 6x6}
       \vspace{-10pt}
\end{wrapfigure}

In 33 classes, the matrix $E_S$ has at least one negative entry and so it is not bistochastic. 
In 31 classes, the maxtrace is attained along an outer permutation, and it exceeds all the inner traces and hence the squared-Frobenius norm of $E_S$. 
In the intersection, we see 6 classes of $E_S$'s with negative entries as well as their maxtraces exceeding the inner traces.

The corresponding numbers for $6\times 6$ matrices are summarized in the adjacent venn diagram in \cref{Fig 6x6}.

\begin{remark}[On simplicial skeletons]
A big chunk of these non-example $E_S$'s seem to occur in simplicial faces of $\Omega_n$
(i.e. when $P_n(S)$ is a linearly independent set).  
For $n=5$, out of the 65 non-examples, 33 are simplicial. In comparison, only 83 of the 534 totally supported skeletons are simplicial. 
See \cite{brualdi1977-1} for a thorough study and characterization of skeletons of simplicial faces, and \cite{Brualdi2021} for simplicial RCDS skeletons.
Using those complete characterizations, it might be worth investigating in higher dimensions the afore-stated phenomenon (and possible explanations, even speculative).
\end{remark}

\begin{remark}[Monte Carlo simulations for $n\geqslant 7$]
We run Algorithm-1 for a random sample of $10000$ skeletons, instead of the whole list from Step 1:
For $n=7$, totally supported skeletons are $4377$ out of which $2758$ are \erdos{} ($63.2\%$). 
For $n=8$, totally supported skeletons are $5713$, with 3476 \erdos{} matrices ($60.8\%$).
\end{remark}

\subsection{Maximum number of distinct entries and denominator} \label{Subsection erdos denominators}
We investigate some interesting subsets of \erdos{} matrices $E$, motivated by their resemblance to {\it magic squares}. For this we look at those $E$'s with all their nonzero entries distinct and with the minimum number of zeros. We shall call this set by $\mathcal{E}^D_n$. There are no such examples up to $n\leqslant 4$. There is only one class of \erdos{} matrices in $\mathcal{E}^D_5$ and 6 classes in $\mathcal{E}^D_6$.

\begin{example} \label{max-den}
Below, $E_5^D$ and $E_{6, 1}^D$ have 18 and 27 distinct nonzero entries respectively.
\begin{equation*}
    E_{5}^{D} = \frac{1}{10226}
    \left(\begin{smallmatrix}
        1028 &\  1947 &\ 3087 &\ 2832 &\ 1332\smallskip\\ 
           0 &\  2204 &\ 3344 &\ 3089 &\ 1589\smallskip\\
        1736 &\  2655 &\ 3795 &\    0 &\ 2040\smallskip\\
        2501 &\  3420 &\    0 &\ 4305 &\    0\smallskip\\
        4961 &\     0 &\    0 &\    0 &\ 5265
    \end{smallmatrix}\right), \ \ 
    E_{6, 1}^{D} = \frac{1}{1499473}
    \left(\begin{smallmatrix}
       280460 &\ 227012 &\ 194276 &\ 321380 &\ 298700 &\ 177645 \smallskip\\
            0 &\ 345696 &\      0 &\ 440064 &\ 417384 &\ 296329 \smallskip\\
       315989 &\ 262541 &\ 229805 &\ 356909 &\ 334229 &\      0 \smallskip\\
       340200 &\ 286752 &\ 254016 &\ 381120 &\      0 &\ 237385 \smallskip\\
            0 &\ 377472 &\ 344736 &\      0 &\ 449160 &\ 328105 \smallskip\\
       562824 &\      0 &\ 476640 &\      0 &\      0 &\ 460009
    \end{smallmatrix}\right).
\end{equation*}
\end{example}

\begin{remark}
In these two dimensions, matrices in $\mathcal{E}^D_n$ turn-out to have the maximum number of distinct entries overall among all \erdos{} matrices of their size. 
Their equivalence classes also attain the full size of $2(n!)^2$; suggestive of high asymmetry. We expect this to hold for higher values of $n$ as well.
It might be interesting to verify (or prove) both the phenomenon for $n\geqslant 7$.
\end{remark}

The two matrices in \cref{max-den} also have the largest denominator among all the \erdos{} matrices (of their orders) when expressed in simplest terms.
This leads us to the next subset of our interest, which we shall denote by $\mathcal{E}^{\max}_n$ of $E$'s that have the largest denominator. Up to dimensions $n\leqslant 6$, these are:\\
$J_2$,\ 
$R=\frac{1}{5}\left(\begin{smallmatrix}3&2&0\\2&1&2\\0&2&3\end{smallmatrix}\right)$\cite[Theorem 4.1]{bouthat2024question},
$\frac{1}{43}\left(\begin{smallmatrix}2&7&15&19\\7&12&0&24\\15&0&28&0\\19&24&0&0\end{smallmatrix}\right)$, and $E_5^D, \ E_{6,1}^D$ from \cref{max-den}.
Their denominators form the sequence:\quad
$2, 5, 43, 10266, 1499473, \ldots$

We believe the matrices in $\mathcal{E}^{\max}_n$ for $n\leqslant 6$ require the most number of permutations in their convex expression \eqref{birkhoff--skeleton-convex-sum} when using minimal number of permutations with nonzero coefficients.
\begin{question}
 In view of above observations, we would like to ask the following natural questions for each $n\geqslant7$:
 1) Is $\mathcal{E}^{\max}_n\subseteq \mathcal{E}^{D}_n$?
 2) How does the above denominator-sequence grow?
\end{question}

\section{Proof of Theorem \ref{Theorem Erdos-ness of 3 families of RCDS}\label{Section proof for RCDS patterns}: \erdos-ness of RCDS-families in \texorpdfstring{\cite{Brualdi2021}}{Brualdi--Dahl}}
Recall the definitions of the RCDS matrix classes in points 1.--3. in \hyperlink{zigzag}{Introduction}.
We begin with
\begin{remark}\label{Remark Brualdi's row size mismatch}
 \cite{Brualdi2021} specifies the number of rows in each of the blocks in $X^{(\vect{s}, \vect{r})}$ containing $\alpha_{j, j}$ and $\alpha_{j, j+1}$ to be $r_{2j-1}$ and resp. $r_{2j}$, but does not comment on $r_{2j-1}\overset{{\bf ?}}{=} r_{2j}$; the same question arises along the columns as well.
If $r_{2j-1}\neq  r_{2j}$ then the block structure of $X$ is perhaps unclear.
Further, even if one fixes $r_{2j-1}=r_{2j}$, $s_{2j}=s_{2j+1}$ $\forall\ j$, the inequalities in the hypothesis \cite[(Eq. 20)]{Brualdi2021} do not guarantee the positivity of $\alpha_{i, j}$'s.
(And $\alpha_{i,j}$ values were not even computed in that paper).
We have made these clear in \eqref{brualdi family 3} (and believe this to be the actual intention of \cite{Brualdi2021}) and the formulas for $\alpha_{i, j}$'s are as follows.
Observe $\alpha_{1, 1} = \frac{1}{r_1},\ \ \alpha_{1, 2}=\frac{r_1-s_1}{r_1\, s_2}, \ \ldots$
\begin{equation}
\alpha_{i, i}=\frac{s_1+\cdots+s_i - r_1-\cdots-r_{i-1}}{r_i\,s_i},\ \ 
\alpha_{i, i+1}=\frac{r_1+\cdots+r_i- s_1-\cdots-s_i}{r_i\,s_{i+1}},\ \ldots
\end{equation}
\end{remark}

\begin{remark} \label{uv expression for Xrs}
This matrix can also be expressed like in \eqref{uv expression}. We first note that by symmetry, all the ${\bf u}_i$'s (and ${\bf v}_j$'s) for a particular block are the same. So we create the vectors $\vect\smolu$ and $\vect\smolv$ by contracting the indices to refer to all the rows and columns of a corresponding block (e.g. $\smolu_1=u_1=\ldots=u_{r_1}$, and so on). With this $\alpha_{i, j}=\smolu_i+\smolv_j$ $\forall$ $1\leqslant i\leqslant j \leqslant i+1\leqslant k+1$. Solving it for other values of $i, j$, yields $\smolu_i + \smolv_j = \sum_{t=i}^{j-2}(\alpha_{t, t+1} - \alpha_{t+1, t+1}) +\alpha_{j-1, j}$ if $j > i+1$ and $\smolu_i + \smolv_j = \sum_{t=j}^{i-1}(\alpha_{t, t} - \alpha_{t, t+1}) +\alpha_{i, i}$ if $j < i$.
\end{remark}

We are now ready to present the proof.

\begin{proof} [Proof of \cref{Theorem Erdos-ness of 3 families of RCDS}]
We know that $X$ in the three families in the theorem have RCDS property.
Thus by \cref{rcds to erdos}, it suffices to show that $X$ when expressed like in \eqref{uv expression}, the value $\smolu_i+\smolv_j \geqslant 0$ for all $i, j$.
\smallskip \\
1. For $X=X^{\big((r_1,r_2),\,(s_1,s_2,s_3)\big)}$ :
So $X=\left(\begin{array}{@{}c|c@{}|c@{}}
    \frac{1}{r_1} \allones{r_1}{s_1} & \frac{r_1-s_1}{r_1s_2} \allones{r_1}{s_2} & {\bf 0}_{r_1 \times s_3} \\ \hline
  {\bf 0}_{r_2\times s_1}   & \frac{s_1+s_2-r_1}{s_2r_2} \allones{r_2}{s_2} & \frac{1}{r_2} \allones{r_2}{s_3}  
\end{array}\right)_{n\times n}$ with $s_1<r_1<s_1+s_2\leqslant r_1+r_2=s_1+s_2+s_3=n$. Next, we can find the vectors $(\smolu_1, \smolu_2)$ and $(\smolv_1, \smolv_2, \smolv_3)$ as discussed in \cref{uv expression for Xrs}. We only need to verify that $\smolu_1 + \smolv_3$ and $\smolu_2 + \smolv_1$ are non-negative and the rest follows from \cref{rcds to erdos}. Indeed, we find
\begin{align*}
    \smolu_1 + \smolv_3 = \alpha_{1, 2} - \alpha_{2, 2} + \alpha_{2, 3} &= n(r_1-s_1)/(r_1r_2s_2)>0,\\
    \smolu_2 + \smolv_1 = \alpha_{1, 1} - \alpha_{1, 2} + \alpha_{2, 2} &= n(s_1+s_2 - r_1)/(r_1r_2s_2)>0.
\end{align*}

\noindent{For $X=X^{(r, s, n)}$} : This case follows from the previous case when we set $s_3=0$.\smallskip

\noindent{For $X=X^{\vect{\alpha}}$} : The following values for the vectors $\vect{u}$ and $\vect{v}$ can be used for \eqref{uv expression} type expansion of $X$:
\begin{align*} \nn
  \smolu_i = \frac{1}{8(\alpha_1\alpha_4+\alpha_2\alpha_3)}\times\begin{cases}
      -\ \alpha_1-\alpha_2+3\alpha_3+3\alpha_4 &:1\leqslant i\leqslant p, \\
      \ 3\alpha_1+3\alpha_2-\ \alpha_3-\ \alpha_4 &:p< i\leqslant 2p.
  \end{cases}\\ 
   \smolv_j = \frac{1}{8(\alpha_1\alpha_4+\alpha_2\alpha_3)}\times\begin{cases}
      -\ \alpha_1+3\alpha_2-\alpha_3+3\alpha_4 &:1\leqslant j\leqslant p, \\
      \ 3\alpha_1-\ \alpha_2+3\alpha_3-\ \alpha_4 &:p< j\leqslant 2p.
  \end{cases}
\end{align*}
Any $\smolu_i+ \smolv_j=\alpha_{l}/(\alpha_1\alpha_4+\alpha_2\alpha_3)>0$ for some $l\in\{1, 2, 3, 4\}$. The proof then follows from \cref{rcds to erdos}.\end{proof}


\begin{remark} \label{type 3 counter}
  The RCDS family $X^{(\vect{s}, \vect{r})}$ contains non-\erdos{} matrices for value of $k \geqslant 3$. As a counter example, consider $X^{((4, 2, 4),\,(3, 2, 5))}$ with
  $\vect{\alpha} = (\frac{1}{4}, \frac{1}{8}, \frac{1}{4}, \frac{1}{10}, \frac{1}{5})$. This matrix is RCDS with all the inner traces equal to $81/40$, but the maxtrace is $82/40$ and so it is not \erdos{}.
\end{remark}
\subsection*{Acknowledgements}
 We deeply thank R. Tripathi and A. Kushwaha for their discussions and feedbacks on this work; and also for suggesting us some references. 
We thank A. Iyyer for sharing with us projects on estimating equivalence classes of binary matrices; and M. Krishnapur for some fruitful discussions.
P. Karmakar thanks ICTS for the excellent research facilities that made this work possible. 
Karmakar also acknowledges the support of the Department of Atomic Energy, Government of India, under project no. RTI4019.
This work was initiated when Souvik Pal was an NBHM Postdoc. (fellowship Ref. No. 0204/9/2024/R\& D-II/2965) at IISc Bangalore. 
G. Krishna Teja acknowledges his NBHM Postdoc. fellowship (Ref. No. 0204/16(8)/2022/R\&D-II/11979), and ISI for the facilities.
We immensely thank the two anonymous referees for their constructive feedback, and clarifications of some concepts, all of which well-improved the exposition of the paper. 
We also thank A. Khare for exposing us to this line of work and for encouragement.

\bibliographystyle{alphaabbr}

\bibliography{ref}

@article{Achilles,
    author = {Achilles, E.},
    title = {Doubly stochastic matrices with some equal diagonal sums},
    journal = {Linear Algebra Appl.},
    year = {1978},
volume = {22},
pages = {293--296},
}

@article{Bala,
    author = {Balasubramanian, K.},
    title = {On equality of some elements in matrices},
    journal = {Linear Algebra Appl.},
    year = {1978},
volume = {22},
pages = {135--138},
}

@article {kushwaha2025note,
    AUTHOR = {Kushwaha, Aman and Tripathi, Raghavendra},
     TITLE = {A note on {E}rd{\H{o}}s matrices and {M}arcus--{R}ee inequality},
   JOURNAL = {Linear Algebra Appl.},
  FJOURNAL = {Linear Algebra and its Applications},
    VOLUME = {725},
      YEAR = {2025},
     PAGES = {223--247},
      ISSN = {0024-3795,1873-1856},
   MRCLASS = {15B51 (15A45 15A80)},
  MRNUMBER = {4933788},
       DOI = {10.1016/j.laa.2025.07.012},
       URL = {https://doi.org/10.1016/j.laa.2025.07.012},
}

@article {tripathi2025erdos,
    AUTHOR = {Tripathi, Raghavendra},
     TITLE = {Some observations on {E}rd{\H{o}}s matrices},
   JOURNAL = {Linear Algebra Appl.},
  FJOURNAL = {Linear Algebra and its Applications},
    VOLUME = {708},
      YEAR = {2025},
     PAGES = {236--251},
      ISSN = {0024-3795,1873-1856},
   MRCLASS = {15A15 (15A45 15B36 15B51)},
  MRNUMBER = {4840450},
MRREVIEWER = {Fr\'ed\'eric\ Morneau-Gu\'erin},
       DOI = {10.1016/j.laa.2024.12.002},
       URL = {https://doi.org/10.1016/j.laa.2024.12.002},
}

@article{Brualdi2021,
     author = {Brualdi, Richard A. and Dahl, Geir},
      title = {Diagonal sums of doubly stochastic matrices},
    journal = {Linear Multilinear Algebra},
   fjournal = {Linear and Multilinear Algebra},
    year = {2021},
volume = {20},
number = {70}, 
pages = {4946–4972}
}

@article{brualdi1977-1,
  title={Convex polyhedra of doubly stochastic matrices. {I}. {A}pplications of the permanent function},
  author={Brualdi, Richard A. and Gibson, Peter M.},
  journal={J. Comb. Theory Ser. A},
  volume={22},
  number={2},
  pages={194--230},
  year={1977},
  publisher={Elsevier}
}

@article{brualdi1977-2,
  title={Convex polyhedra of doubly stochastic matrices: {II}. {G}raph of {$\Omega_n$}},
  author={Brualdi, Richard A. and Gibson, Peter M.},
  journal={J. Comb. Theory Ser. B},
  volume={22},
  number={2},
  pages={175--198},
  year={1977},
  publisher={Elsevier}
}

@article{brualdi1977-3,
  title={Convex polyhedra of doubly stochastic matrices {III}. {A}ffine and combinatorial properties of {$\Omega_n$}},
  author={Brualdi, Richard A. and Gibson, Peter M.},
  journal={J. Comb. Theory Ser. A},
  volume={22},
  number={3},
  pages={338--351},
  year={1977},
  publisher={Elsevier}
}

@article{brualdi1976-4,
  title={Convex polyhedra of doubly stochastic matrices--{IV}},
  author={Brualdi, Richard A. and Gibson, Peter M.},
  journal={Linear Algebra Appl.},
  volume={15},
  number={2},
  pages={153--172},
  year={1976},
  publisher={Elsevier}
}

@article {marcus1959diagonals,
    AUTHOR = {Marcus, M. and Ree, R.},
     TITLE = {Diagonals of doubly stochastic matrices},
   JOURNAL = {Quart. J. Math. Oxford Ser. (2)},
  FJOURNAL = {The Quarterly Journal of Mathematics. Oxford. Second Series},
    VOLUME = {10},
      YEAR = {1959},
     PAGES = {296--302},
      ISSN = {0033-5606,1464-3847},
   MRCLASS = {15.00},
  MRNUMBER = {117243},
MRREVIEWER = {L.\ Mirsky},
       DOI = {10.1093/qmath/10.1.296},
       URL = {https://doi.org/10.1093/qmath/10.1.296},
}

@article {birkhoff1946tres,
    AUTHOR = {Birkhoff, Garrett},
     TITLE = {Tres observaciones sobre el algebra lineal},
TRANSTITLE = {Three observations on linear algebra},
   JOURNAL = {Univ. Nac. Tucum\'an. Revista A.},
  FJOURNAL = {Universidad Nacional de Tucum\'an. Revista A.},
    VOLUME = {5},
      YEAR = {1946},
     PAGES = {147--151},
   MRCLASS = {09.1X},
  MRNUMBER = {20547},
MRREVIEWER = {J.\ L.\ Dorroh},
}

@article {bouthat2024question,
    AUTHOR = {Bouthat, Ludovick and Mashreghi, Javad and Morneau-Gu\'erin,
              Fr\'ed\'eric},
     TITLE = {On a question of {E}rd{\H{o}}s on doubly stochastic matrices},
   JOURNAL = {Linear Multilinear Algebra},
  FJOURNAL = {Linear and Multilinear Algebra},
    VOLUME = {72},
      YEAR = {2024},
    NUMBER = {17},
     PAGES = {2823--2844},
      ISSN = {0308-1087,1563-5139},
   MRCLASS = {15B51 (15A15)},
  MRNUMBER = {4823250},
MRREVIEWER = {Pietro\ Paparella},
       DOI = {10.1080/03081087.2023.2300674},
       URL = {https://doi.org/10.1080/03081087.2023.2300674},
}

@article{Erorychev,
    author = {Erorychev, G. P.},
    title = {Proof of the van der Waerden conjecture for permanents},
    journal = {Sib. Math. J.},
volume = {22},
    year = {1981},
pages = {854--859},
doi = {10.1007/BF00968054},
url = {https://doi.org/10.1007/BF00968054},
}

@article{Hwang,
    author = {Hwang, S. G.},
    title = {Minimum permanent on faces of staircase type of the polytope of doubly stochastic matrices},
    journal = {Linear Multilinear Algebra},
   fjournal = {Linear and Multilinear Algebra},
    year = {1985},
volume = {4},
pages = {271--306},
}

@article{BrualdiMinper,
  title={Minimum permanents on special faces of the polytope of doubly stochastic matrices},
  author={Richard A. Brualdi and Bryan L. Shader},
  journal={Linear Algebra Appl.},
  year={1994},
  volume={201},
  pages={103--111},
  }

@misc{oeisA326342,
    Author = {{OEIS Foundation Inc. (2025)}},
    Note = {Entry {A326342} in the {OEIS}, \url{https://oeis.org/A326342}},
    Title = {Number of {$n\times n$} binary matrices with total support},
    Year = {2025},
}

@misc{oeisA326343,
    Author = {{OEIS Foundation Inc. (2025)}},
    Note = {Entry {A326343} in the {OEIS}, \url{https://oeis.org/A326343}},
    Title = {Number of inequivalent {$n\times n$} binary matrices with total support, where equivalence means permutations of rows or columns},
    Year = {2025},
}

@misc{oeisA381896,
    Author = {{OEIS Foundation Inc. (2025)}},
    Note = {Entry {A381896} in the {OEIS}, \url{https://oeis.org/A381896}},
    Title = {Number of {$n\times n$} {E}rd{\H{o}}s matrices up to equivalence},
    Year = {2025},
}

@article {Sinkhorn,
    AUTHOR = {R. Sinkhorn and P. Knopp},
     TITLE = {Concerning nonnegative matrices and doubly stochastic matrices},
   JOURNAL = {Linear Multilinear Algebra},
  FJOURNAL = {Linear and Multilinear Algebra},
    VOLUME = {2},
      YEAR = {1967},
    NUMBER = {21},
     PAGES = {343--348},
       DOI = {10.2140/pjm.1967.21.343},
       URL = {https://doi.org/10.2140/pjm.1967.21.343},
}

@article{Sinkhorndiagonals ,
  title={Doubly stochastic matrices which have certain diagonals with constant sums},
  author={Sinkhorn, Richard},
  journal={Linear Algebra Appl.},
  volume={16},
  number={1},
  pages={79--82},
  year={1977},
  publisher={Elsevier}
}

@misc{github,
    Author = {Krishna, Hariram},
    Note = {Repository in Github, \url{https://github.com/Harirarn/Erdos_matrics/}},
    Title = {Erdos{\_}matrics},
    Year = {2025},
}
\bigskip

\newcommand{\author}[1]{\def\curauthor{#1}(#1)}
\newcommand{\address}[1]{\textsc{#1}}
\newcommand{\email}[1]{\href{mailto:"\curauthor"<#1>}{\texttt{#1}}}
\ExplSyntaxOn
\NewDocumentCommand{\emails}{>{\SplitList,}m}{\textit{Email: }\def\tempcomma{\ }
\tl_map_inline:nn{#1}{\tempcomma\email{##1}\def\tempcomma{,\ }}.
\smallskip}\ExplSyntaxOff
  \noindent
\author{Priyanka Karmakar}
  \address{ICTS-TIFR, Bangalore 560 089, India.}\\
  \emails{priyankakarmakar.2196@gmail.com, priyanka.karmakar@icts.res.in}\\
\author{Hariram Krishna}
  \address{Department of Mathematics, IISc, Bangalore 560 012, India.}\\
  \emails{hariramk@iisc.ac.in, harirarn@gmail.com}\\
\author{Souvik Pal}
  \address{Department of Sciences and Humanities, Christ University, Bangalore 560 074, India.}
  \emails{pal.souvik90@gmail.com, souvik.pal@christuniversity.in}\\
\author{G. Krishna Teja}
  \address{Stat. Math. Unit, Indian Statistical Institute Bangalore Center, Bangalore, 560 059, India.}
  \emails{tejag@alum.iisc.ac.in, tejag\_pd@isibang.ac.in}

\end{document}